\documentclass{amsart}

\usepackage{amssymb,latexsym,amsmath}
\usepackage{graphicx, hyperref}

\newtheorem*{acknowledgement}{Acknowledgements}
\newtheorem{theorem}{Theorem}[section]

\newtheorem{proposition}[theorem]{Proposition}
\newtheorem{remark}[theorem]{Remark}

\newtheorem{corollary}[theorem]{Corollary}

\newcommand{\ga}{\gamma}

\newcommand{\e}{\varepsilon}

\newcommand{\wh}{\widehat}

\newcommand{\ZR}{\mathbb{R}}
\newcommand{\R}{\mathbb{R}}
\newcommand{\ZZ}{\mathbb{Z}}

\newcommand{\eit}{e^{i t \Delta}}

\newcommand{\norm}[1]{ \| #1 \|}

\begin{document}

\title[Lower bounds for the Schr\"odinger maximal function]{Lower bounds for estimates of the Schr\"odinger maximal function}

\author[X. Du]{Xiumin Du}
\address{
University of Maryland\\
College Park, MD}
\email{xdu@math.umd.edu}

\author[J. Kim]{Jongchon Kim}
\address{
University of British Columbia\\
Vancouver, BC}
\email{jkim@math.ubc.ca}

\author[H. Wang]{Hong Wang}
\address{
Massachusetts Institute of Technology\\
Cambridge, MA}
\email{hongwang@mit.edu} 

\author[R. Zhang]{Ruixiang Zhang}
\address{
University of Wisconsin-Madison\\
Madison, WI}
\email{ruixiang@math.wisc.edu}

\begin{abstract}
We give new lower bounds for $L^p$ estimates of the Schr\"odinger maximal function by generalizing an example of Bourgain. 
\end{abstract}

\maketitle

\section{Introduction}
\setcounter{equation}0

\noindent Let
$$
  e^{it\Delta}f(x)=(2\pi)^{-n/2}\int e^{i\left(x\cdot\xi+t|\xi|^2\right)}\widehat{f}(\xi) \, d\xi
$$ denote the solution to the free Schr\"{o}dinger equation
\begin{equation*}
  \begin{cases}
    iu_t - \Delta u = 0, &(x,t)\in \mathbb{R}^n \times \mathbb{R} \\
    u(x,0)=f(x), & x \in \mathbb{R}^n.
  \end{cases}
\end{equation*}
We are interested in the value of $\bar \ga_{n,p}$, the infimum of the numbers $\ga_{n,p}$ such that the following Schr\"{o}dinger maximal estimate holds:
\begin{equation}\label{Lp-R}
\left\| \sup_{0 < t \le R} | e^{it \Delta} f| \right\|_{L^p(B^n(0,R))}  \lessapprox R^{\gamma_{n,p}} \| f \|_{L^2}\,, \quad \forall f: {\rm supp}\wh f \subset B^n(0,1)\,.
\end{equation}
Here $A\lessapprox B$ denotes $A\leq C_\e R^\e B$ for some constant $C_\epsilon>0$ for any $\e>0,R>1$. We also write $A \gtrsim B$ if $A\geq CB$ for an absolute constant $C>0$.

Estimates of the form \eqref{Lp-R}, especially the case $p=2$, have applications to Carleson’s pointwise convergence problem for Schr\"{o}dinger solutions \cite{lC} and have been studied extensively by many authors. The state-of-art results are summarized as follows. Due to examples by Dahlberg--Kenig \cite[$n=1$]{DK} and Bourgain \cite[$n\geq 2$]{jB16}, and positive results by Kenig--Ponce--Vega \cite[$n=1$]{KPV}, D.--Guth--Li \cite[$n=2$]{DGL} and D.--Z. \cite[$n\geq 3$]{DZ}, it is known that 
\begin{equation} \label{eq-ga}
  \bar \ga _{n,p}=\max\left\{n\left(\frac 1 p -\frac{n}{2(n+1)}\right),0\right\}  
\end{equation}
for any $p\geq 1$ when $n=1,2$, and $1\leq p\leq 2$ when $n\geq 3$. Also, from the Stein-Tomas Fourier restriction theorem it follows that $\bar \ga_{n,p}=0$ for $p\geq \frac{2(n+2)}{n}$. However, it remains as an interesting problem to determine $\bar\ga_{n,p}$ for $2<p<\frac{2(n+2)}{n}$ when $n\geq 3$. 

It may seem plausible that \eqref{eq-ga} should hold for any $p\geq 1$ and $n\geq 1$. However, we disprove this for a certain range of $p$ when $n\geq 3$.
Our main result is the following lower bound for $\bar\ga_{n,p}$.
\begin{theorem} \label{thm-ga} Let $n\geq 3$ and $p\geq 2$. For every integer $1\leq  m \leq n$,
\[ \bar\ga_{n,p}\geq \frac{n+m}{2}\left(\frac 1 p-\frac 1 2 \right)+\frac{m}{2(m+1)}.\]
\end{theorem}

The example that proves Theorem \ref{thm-ga} is built upon Bourgain's example \cite{jB16} that provides the lower bound for the case $m=n$. For the case $1\leq m< n$, we take Bourgain's example in the intermediate dimension $m$ and then ``fatten" it to a function on $\R^n$. 

We state two special cases of Theorem \ref{thm-ga} as a corollary.
\begin{corollary}
If $\bar\ga_{n,p}=n(\frac 1 p-\frac{n}{2(n+1)})$, then $$p\leq p_0(n):=2+\frac{4}{(n-1)(n+2)}.$$
If $\bar\ga_{n,p}=0$, then $$p\geq p_1(n):= \max _{m\in \ZZ, 1\leq m\leq n} 2+\frac{4}{n-1+m+n/m}.$$
\end{corollary}

\begin{remark}
Note that $p_0(n)<\frac{2(n+1)}{n}<p_1(n)$ when $n\geq 3$. Therefore, \eqref{eq-ga} fails for $p_0(n)<p<p_1(n)$ when $n\geq 3$.
\end{remark}

Finally, we remark that some upper bounds for $\bar\ga_{n,p}$ can be obtained from weighted Fourier restriction estimates, c.f. \cite{DZ}. In particular, we refer the reader to \cite{DGOWWZ} for such estimates with $p=2(n+1)/n$, which was obtained via the polynomial partitioning method \cite{G1,G2} and refined Strichartz estimates \cite{DGL, DGLZ}. 
For $p>2(n+1)/n$, one can get new upper bounds by using an additional ingredient, the fractal $L^2$ restriction estimate \cite{DZ}. However, it seems that new ingredients are still needed to get sharp results. We do not explore along this direction in the current paper.

\begin{acknowledgement} 
This work was initiated at the AMS 2018 Mathematics Research Communities (MRC)  program ``Harmonic Analysis: New Developments on Oscillatory Integrals". We wish to thank the organizers for the fruitful program.

This material is based upon work supported by the National Science Foundation under Grant Number DMS 1641020. The first, second and fourth authors were supported in part by the National Science Foundation under Grant Number DMS 1638352. They were additionally supported by the Shiing-Shen Chern Fund, a PIMS postdoctoral fellowship and the James D. Wolfensohn Fund, respectively.
\end{acknowledgement}

\section{An example that proves Theorem \ref{thm-ga}}
\setcounter{equation}0

Theorem \ref{thm-ga} is a consequence of the following.
\begin{proposition} \label{prop}
Let $m, n$ be integers with $1\leq m \leq  n$. For any $R>1$, there exists $f \in L^2 (\ZR^n)$ with $\wh f$ supported in the annulus $\{ \xi \in \ZR^n: |\xi| \sim R \}$ satisfying the following property; There is a set $E\subset B^n(0,1)$ of measure comparable to $R^{-\frac{n-m}{2}}$ such that for every $x\in E$, 
\[
\frac{ |e^{it\Delta}f(x)|} { \norm{f}_{L^2}} \gtrsim R^{\frac{m}{2(m+1)}} R^{\frac{n-m}{4}}\;\;\text{ for some } \; t = -\frac{x_1}{2R} + O(R^{-3/2}).
\]
\end{proposition}
\begin{proof}
We write $\bar x=(x,x')\in \R^m \times \R^{n-m}$ and $\bar \xi =(\xi, \xi')\in \R^m \times \R^{n-m}$. 

We briefly recall an estimate for the example $f_0\in L^2(\ZR^m)$ from \cite{jB16}, where $\widehat{f_0}$ is supported in the annulus $\{ \xi \in \ZR^m: |\xi| \sim R \}$; There is a set $E_0 \subset B^m(0,1)$ of measure comparable to 1 such that for every $x\in E_0$, 
\begin{equation}\label{eq:BG}
\frac{ |e^{it\Delta}f_0(x)|} { \norm{f_0}_{L^2}} \gtrsim R^{\frac{m}{2(m+1)}} \;\;\text{for some} \; t = -\frac{x_1}{2R} + \tau \text{ with } \;  |\tau| \leq \frac{1}{10}R^{-3/2}.
\end{equation}
See also \cite{LR} for a different example based on \cite{BBCRV}, which provides an estimate essentially the same as \eqref{eq:BG}. 

Let $\chi=\chi_{[-\frac 1 2,\frac 1 2]}$ be the characteristic function of the interval $[-\frac 1 2,\frac 1 2]$. Let $f_1(x')$ be given by 
\[ \widehat{f_1} (\xi') = \prod_{j=m+1}^n R^{-\frac{1}{4} }\chi \left(R^{-\frac 1 2} (\xi_j - R) \right), \]
so that $\norm{f_1}_{L^2(\R^{n-m})} = 1$. The choice of the function $f_1$ is motivated by the example from \cite{jB16}. Note that 
\[ |\eit f_1 (x')| = (2\pi)^{-(n-m)/2} \prod_{j=m+1}^n R^{\frac{1}{4} } \left| \int_{[-\frac 1 2,\frac 1 2]} e^{i( R^{1/2} \xi_j(x_j+2Rt) + tR\xi_j^2)} d\xi_j \right|.  \]  
When $|t+\frac{x_1}{2R}| \leq  \frac{1}{2}R^{-3/2}$ and $|x_j - x_1| \leq \frac{1}{2}R^{-1/2}$ for each $m< j\leq n$, there is little cancellation in the above integral and therefore
\begin{equation}\label{eq:f1}
|\eit f_1 (x')| \gtrsim R^{\frac{n-m}{4}}.
\end{equation} 

We take $f$ to be the tensor product of $f_0$ and $f_1$, i.e., \[ f(\bar x)  := f_0(x) f_1 (x').\] 
Let $E$ be the set given by
\[ E = \{ (x,x') \in B^n(0,1) : x \in E_0 \text{ and } \max_{m< j\leq n} |x_j- x_1| \leq \frac{1}{2}R^{-1/2} \}. \]
It follows that the measure of the set $E$ is comparable to $R^{-\frac{n-m}{2}}$. Moreover, for any $\bar x = (x,x') \in E$, we have by \eqref{eq:BG} and \eqref{eq:f1}, 
\[ \frac{|\eit f(\bar x)|}{\norm{f}_{L^2}} = \frac{|\eit f_0( x)|}{\norm{f_0}_{L^2}} \frac{|\eit f_1( x')|}{\norm{f_1}_{L^2}} \gtrsim R^{\frac{m}{2(m+1)}} R^{\frac{n-m}{4}} \] 
for some $t$ satisfying $|t+\frac{x_1}{2R}| \leq \frac{1}{10} R^{-3/2}$.
\end{proof}

We proceed to the proof of Theorem \ref{thm-ga}. It follows from Proposition \ref{prop} that, 
\begin{equation}\label{eq:Lp}
\bigg\|\sup_{0<t\leq \frac 1 R}\left|\eit f\right|\bigg\|_{L^p(B^n(0,1))}  \gtrsim R^{\frac{m-n}{2} (\frac{1}{p}-\frac{1}{2}) } R^{\frac{m}{2(m+1)}} \|f\|_2.
\end{equation}
Theorem \ref{thm-ga} follows from \eqref{eq:Lp} by scaling. Define the function $g\in L^2(\ZR^n)$ by
\begin{equation*}
\wh g (\xi) = R^{\frac n 2}  \wh f (R\xi)
\end{equation*}
so that $\widehat{g}$ is supported in the annulus $|\xi| \sim 1$ and $\norm{g}_{L^2} = \norm{f}_{L^2}$. By parabolic rescaling, we have
\begin{equation*}
|\eit f(x)| = R^{\frac n 2} |e^{i R^2 t \Delta} g(R x)|.
\end{equation*}
Hence, by \eqref{eq:Lp}, 
\begin{align*}
    \bigg\| \sup_{0<t\leq R} |e^{i t \Delta} g| \bigg\|_{L^p(B^n(0,R))} 
    &= R^{n(\frac 1 p - \frac 1 2)} \bigg\| \sup_{0<t\leq \frac 1 R } |e^{i t \Delta} f| \bigg\|_{L^p(B^n(0,1))} \\
    &\gtrsim R^{\frac{n+m}{2} (\frac{1}{p}-\frac{1}{2}) } R^{\frac{m}{2(m+1)}} \|g\|_2.
\end{align*}
This finishes the proof of Theorem \ref{thm-ga}.

\end{document}